\documentclass[12pt]{amsart}
\usepackage{amsfonts,amssymb,amscd,xy}
\usepackage[leqno]{amsmath}
\usepackage{mathrsfs}
\usepackage{amssymb}
\usepackage{amsmath}
\usepackage{amsxtra}
\usepackage{amsthm}

\usepackage[usernames,dvipsnames]{xcolor}
\usepackage{tabularx}
\usepackage{epsfig,amssymb}
\usepackage{bm} 
\usepackage{verbatim} 

\usepackage{ulem}

\usepackage{hyperref}
\definecolor{labelkey}{rgb}{1,0,0}

\usepackage{verbatim} 

\makeatletter

\DeclareMathAlphabet{\mathcalligra}{T1}{calligra}{m}{n}

\numberwithin{equation}{section}

\def\der{{\rm{der}}}
\def\tor{{\lle\rm{tor}}}

\newcommand{\sh}{\kern -.4em\phantom{a}^{\mathbf{\sim}}}

\newcommand{\g}{\textsl{g}}

\newcommand{\kbar}{\overline{k}}
\def\be{\kern -.1em}
\def\le{\kern 0.03em}
\def\lle{\kern 0.04em}
\def\lbe{\kern -.025em}

\newcommand{\Z}{{\mathbb Z}}
\newcommand{\Q}{{\mathbb Q}}

\newcommand{\G}{{\mathbb G}}

\newcommand{\spec}{\mathrm{ Spec}\,}

\newcommand{\Hom}{{\mathrm{Hom}}}
\newcommand{\Ext}{{\mathrm{Ext}}}

\newcommand{\Knr}{K^{\le\rm{nr}}}

\def\e{\kern 0.08em}

\newcommand{\s}{\mathscr }

\newcommand{\mr}{\mathrm }

\newcommand{\sep}{\le{\rm sep}}

\newtheorem{lemma}{Lemma}[section]
\newtheorem{theorem}[lemma]{Theorem}
\newtheorem{proposition-definition}[lemma]{Proposition-Definition}
\newtheorem{corollary}[lemma]{Corollary}
\newtheorem{proposition}[lemma]{Proposition}
\theoremstyle{definition}

\theoremstyle{remark}
\newtheorem{remark}[lemma]{Remark}
\newtheorem{remarks}[lemma]{Remarks}

\definecolor{labelkey}{rgb}{1,0,0}

\begin{document}

\input xy     
\xyoption{all}

\title[The cohomology of tori over local fields]{On the cohomology of tori over local fields with perfect residue field}

\date{\today}


\author{Alessandra Bertapelle}
\address{Dipartimento di Matematica, Universit\`a di Padova, via Trieste 63, I-35121 Padova}
\email{alessandra.bertapelle@unipd.it}

\author{Cristian D. Gonz\'alez-Avil\'es}
\address{Departamento de Matem\'aticas, Universidad de La Serena. Cisternas 1200, La Serena 1700000, Chile}
\email{cgonzalez@userena.cl}

\thanks{C.\,G-A. was partially supported by Fondecyt grant 1120003.}


\topmargin -1cm

\smallskip

\begin{abstract}  If $T$ is an algebraic torus defined over a discretely valued field $K$ with perfect residue field $k$, we relate the $K$-coho\-mo\-logy of $T$ to the $k$-cohomology of certain objects associated to $T$. When $k$ has cohomological dimension $\leq 1$, our results have a particularly simple form and yield, more generally, isomorphisms between Borovoi's abelian $K$-cohomology of a reductive group $G$ over $K$ and the $k$-cohomology of a certain quotient of the algebraic fundamental group of $G$.
\end{abstract}

\maketitle

\section{Introduction}
Let $A$ be a complete discrete valuation ring with field of fractions $K$ and {\it perfect} residue field $k$. Let $\kbar$ and $K^{\sep}$ be fixed separable algebraic closures of $k$ and $K$, respectively, and let $\g$ and $\mathcal G$ denote the corresponding absolute Galois groups. Further, let $\Knr$ denote the maximal unramified extension of $K$ in $K^{\sep}$ and let $I=\text{Gal}\lbe\big(K^{\sep}\be/\Knr\le\big)$ be the inertia subgroup of $\mathcal G$. Then there exists a canonical isomorphism of groups $\mathcal G/I=\g$. If $M$ is a $\mathcal G$-module, $M_{\lbe I}$ will denote the $\g$-module of $I$-coinvariants of $M$.
Now let $T$ be a $K$-torus and let $X^{\lbe\ast}\be(T\e)$ (respectively, $X_{\lbe\ast}\lbe(T\e)$) denote the $\mathcal G$-module of characters (respectively, cocharacters) of $T$. Let $\s T$ denote the N\'eron model of $T$ over $S:=\spec A$ and let $i\colon\spec k \to S$ be the canonical closed immersion. The {\it group of components of $\s T$}, i.e., the (continuous) $\g$-module $\phi(T\e)$ which corresponds to the \'etale $k$-sheaf $i^{*}\be\big(\be\s T/\s T^{\e 0}\le\big)$, was described by Xarles in \cite{xa} in terms of $X^{\lbe\ast}\be(T\e)$. The description given in \cite{xa} is simple when $\phi(T\e)$ is either torsion or torsion-free, but this is not the case in general. When $k$ is finite, a much simpler description of $\phi(T\e)$ was obtained by Bitan in \cite[(3.1)]{bit}, who showed the existence of an isomorphism of $\g$-modules
$\phi(T\e)\simeq X_{\lbe\ast}\lbe(T\e)_{\lbe I}$ for such $k$.
We generalize Bitan's beautiful formula to the case of any perfect field $k$. That is, we prove
\begin{theorem}\label{phit} There exists a canonical isomorphism of $\g$-modules
\[
\phi(T\e)\overset{\!\sim}{\to}X_{\lbe\ast}\lbe(T\e)_{\lbe I}\e.
\]
\end{theorem}

We note that Bitan  \cite[(3.1)]{bit} obtained his formula by combining work of Kottwitz
\cite[ \S7.2]{kott} and of Haines and Rapoport \cite[Appendix]{pr}.
Since \cite[Appendix]{pr} depends on Bruhat-Tits, so does the proof of \cite[(3.1)]{bit}.
Although Bitan's method can be extended to yield a proof of his formula for any perfect field $k$, in this paper we have chosen to generalize his formula by means of an {\it explicit and functorial construction} of the  isomorphism of Theorem \ref{phit} independently of Bruhat-Tits theory.

The above theorem has a number of (immediate) consequences which shed new light on the present subject. For example, the theorem implies that the functor $\phi(-)$ transforms short exact sequences of $K$-tori into 6-term exact sequences of $\g$-modules. See Proposition \ref{6-term} for the precise statement.

In Section \ref{t-coh} we use Theorem \ref{phit} to relate the $K$-cohomology of $T$ to the $k$-cohomology of $X_{\lbe\ast}\lbe(T\e)_{\lbe I}$. When $k$ has cohomological dimension $\leq 1$, our results have the following simple form.

\begin{theorem}[=Theorem \ref{kt2}]  Assume that $k$ has cohomological dimension $\leq 1$. Then, for $r=1$ and $2$, there exist canonical isomorphisms of abelian groups
\[
 H^{\le r}(K,T\e)\simeq  H^{\le r}\be\big(k,X_{\lbe\ast}\lbe(T\e)_{\lbe I}\lbe\big).
\]
If $r\geq 3$, the groups $ H^{\le r}(K,T\e)$ vanish.
\end{theorem}

In Section \ref{ab-coh} we generalize the above theorem from $K$-tori to arbitrary connected reductive algebraic $K$-groups $G$. More precisely, let $\pi_{1}(G\le)$ be the algebraic fundamental group of $G$. Then the following holds.

\begin{theorem} [=Theorem \ref{abt}] Assume that $k$ has cohomological dimension $\leq 1$ and let $G$ be a connected reductive algebraic group over $K$. Then, for $r=1$ and $2$, there exist isomorphisms of abelian groups
\[
H^{\le r}_{\rm{ab}}(K_{\rm{fl}},G\e)\simeq  H^{\le r}\be\big(k,\pi_{1}(G\le)_{\lbe I}\lbe\big).
\]
If $r\geq 3$, the groups $ H^{\le r}_{\rm{ab}}(K_{\rm{fl}},G\e)$ vanish.
\end{theorem}

\begin{corollary} [=Corollary \ref{abcor}] Assume that $k$ has cohomological dimension $\leq 1$ and let $G$ be a connected reductive algebraic group over $K$. Then there exists a bijection of pointed sets
\[
H^{1}(K,G\le)\simeq H^{\le 1}(k,\pi_{1}(G\le)_{\lbe I}\lbe).
\]
In particular, $H^{1}(K,G\le)$ can be endowed with an abelian group structure.
\end{corollary}

\section*{Acknowledgements} We are very grateful to Mikhail Borovoi for some valuable suggestions and for sending us the proof of Lemma \ref{bor}. We also thank Rony Bitan for helpful comments.

\section{Preliminaries}

\subsection{The Basic Setting}
We keep the notation introduced in the previous section.
If $K$ is any field and $T$ is a $K$-torus, let $\underline{X}^{\ast}\be(T\e):=\underline{\Hom}_{\e K}(T,\G_{m,K})$ be the \'etale $K$-sheaf of characters of $T$ and set  $X^{\ast}\be(T\e)=\underline{X}^{\ast}\be(T\e)(K^{\sep})$. Note that, since  $\underline{X}^{\ast}\be(T\e)$ is locally constant, it is represented by a unique (commutative) \'etale $K$-group scheme. See \cite[Proposition II.9.2.3, p.~153]{t}. There exists a canonical isomorphism of \'etale $K$-sheaves
\begin{equation}\label{tor}
T=\underline{\Hom}_{\e K}\lbe(\e\underline{X}^{\ast}\be(T\e),\G_{m,K}).
\end{equation}
See, for example, \cite[Theorem 0.3.12, p.~11]{bra}. In particular, if $L/K$ is any Galois subextension of $K^{\sep}/K$ which splits $T$, then there exists a canonical isomorphism of $\text{Gal}(L/K)$-modules 
\begin{equation}\label{thom}
T(L)=\Hom\e(X^{\ast}\be(T),L^{\be\ast}).
\end{equation}
Now let $\underline{X}_{\le\ast}\lbe(T\e):=\underline{\Hom}_{\e K}(\G_{m,K},T\e)$ be the \'etale $K$-sheaf of cocharacters of $T$ and set $X_{\ast}\lbe(T\e)=\underline{X}_{\le\ast}\lbe(T\e)(K^{\sep})$. We will identify $\underline{X}_{\le\ast}\lbe(T\e)$ with $\underline{X}^{\ast}\be(T)^{\vee}:=\underline{\Hom}_{\e K}\e(\underline{X}^{\ast}\be(T),\Z_{K})$ and $X_{*}\lbe(T\e)$ with
$X^{*}\be(T)^{\vee}=\Hom(X^{*}\be(T),\Z)$.

Let $S$ be a scheme and let $(\mr{Sch}/S\e)^{\,\sh}_{\rm{\acute{e}t}}$ be the category of sheaves of sets on the \'etale site over $S$. By \cite[I, 1.1]{sga3} (see also \cite[Theorem II.3.1.2, p.~97]{t}), the functor $h_{\lbe S}\colon (\mr{Sch}/S\e)\to (\mr{Sch}/S\e)^{\,\sh}_{\rm{\acute{e}t}},\,Y\mapsto \Hom_{S}(-,Y\le)$, is fully faithful. If $Y$ is an $S$-scheme, we will identify $Y$ with $h_{\lbe S}(Y)$, i.e., with the \'etale sheaf it represents. If $S=\spec F$, where $F$ is a field, then 
\begin{equation}\label{ff}
h_{F}\colon (\mr{Sch}/F\e)\to (\mr{Sch}/F\e)^{\,\sh}_{\rm{\acute{e}t}},\quad Y\mapsto \Hom_{F}(-,Y\le),
\end{equation}
is, in fact, an equivalence of categories by \cite[p.~54, last paragraph]{mi1}.

Now recall $S=\spec A$ and $i\colon\spec k \to S$. Let $j\colon\spec K\to S$ be the canonical morphism. The N\'eron model $\s T$ of $T$ over $S$ is a smooth and
separated $S$-group scheme which represents the sheaf
$j_{\e *}T$ on the \'etale (in fact, small smooth) site over $S$. See \cite[Proposition 10.1.6, p.~292]{blr}. With one exception (namely, in Proposition \eqref{phi-ex0}(i)), we will regard $j_{\e *}T$ as an \'etale sheaf on $S$ and identify it with (the \'etale sheaf on $S$ represented by) $\s T$. Thus we may write $j_{*}T=\s T$. The identity component $\s T^{\e 0}$ of $\s T$ is a smooth affine $S$-group
scheme of finite type. See \cite[Proposition 3, p.~18]{km} and \cite[$\text{VI}_{\text{B}}$, Corollary 3.6]{sga3}. Now set $\s T_{\lbe s}=\s T\times_{S}\spec k$ and $\s T_{\lbe s}^{\e 0}=\s T^{\e 0}\times_{S}\spec k$. Then $\s T_{s}^{\e 0}$ is a smooth, connected and affine $k$-group scheme of finite type (see \cite[Proposition 17.3.3(iii)]{ega4.4}, \cite[Proposition 1.6.2(iii)]{ega2} and \cite[Proposition 6.3.4(iii), p.~304]{ega1}). By
\cite[II, \S5, no.1, Proposition 1.8, p.~237]{dg}, the \'etale $k$-group scheme
$\pi_{0}(\s T_{\lbe s}):=\s T_{\lbe s}/\le\s T_{s}^{\e 0}$ has the following universal property: if $E$ is an \'etale $k$-group scheme and $\s T_{\lbe s}\to E$ is a homomorphism of $k$-group schemes, then there exists a unique homomorphism of $k$-group schemes $\pi_{0}(\s T_{\lbe s})\to E$ such that the following diagram commmutes
\begin{equation}\label{diag1}
\xymatrix{\s T_{\lbe s}\ar@{->>}[r]\ar[dr]& \pi_{0}(\s T_{\lbe s})\ar[d]\\
&E.&
}
\end{equation}
The $k$-group schemes $\s T_{\lbe s}$ and $\s T_{s}^{\e 0}$ represent the \'etale $k$-sheaves $i^{*}\s T$ and $i^{*}\s T^{\e 0}$, i.e., we may write $i^{*}\s T=\s T_{\lbe s}$ and $i^{*}\s T^{\e 0}=\s T_{s}^{\e 0}$ as \'etale sheaves on $\spec k$. Then the \'etale $k$-sheaf $\phi(T\e):=i^{*}\lbe\big(\lbe\s T/\s T^{\e 0}\e\big)$ is represented by $\pi_{0}(\s T_{\lbe s})$, i.e,
\[
\phi(T\e)=\pi_{0}(\s T_{\lbe s})
\]
as \'etale sheaves on $\spec k$. We will often identify the \'etale sheaf $\phi(T\e)$ with the $\g$-module $\phi(T\e)\big(\e\kbar\e\big)=\pi_{0}(\s T_{\lbe s})\big(\e\kbar\e\big)$ (see \cite[Corollary II.2.2(i), p.~94]{t}). By \cite[Theorem 2.3.2, p.~51]{bra}, $\phi(T\e)$ is a finitely generated $\g$-module. Now, since $j^{*}\be\big(\be\s T/\s T^{\e 0}\le\big)=0$, there exists a canonical isomorphism of \'etale sheaves
$i_{*}\phi(T\e)=\s T/\s T^{\e 0}$ (see \cite[proof of Theorem II.8.1.2, p.~135]{t}). Thus there exists a canonical exact sequence of \'etale sheaves on $S$
\begin{equation}\label{pseq}
0\to\s T^{\e 0}\to\s T\to i_{*}\phi(T\e)\to 0.
\end{equation}
If $T=\G_{m,K}$, the preceding sequence is
\begin{equation}\label{pseq2}
0\to\G_{m,S}\to j_{*}\G_{m,K}\to i_{*}\Z_{\e k}\to 0,
\end{equation}
where the right-hand nontrivial morphism is induced by the valuation $v\colon K^{*}\to \Z$. See \cite[ \S10.1, Example 5, p.~291]{blr}.

\subsection{Group Cohomology}

Let $J$ be a finite group. We will write $|J\le|$ for its order and $\frak A_{\be J}$ for the augmentation ideal of $\Z[J\e]$, i.e., the kernel of the homomorphism $\Z[J\e]\to\Z,\,
\Sigma \,n_{\sigma}\sigma\mapsto \Sigma\, n_{\sigma}$. If $M$ is a finitely generated (left) $J$-module, $M_{J}:=M/\frak A_{\be J}M$ is the largest quotient of $M$ on which $J$ acts trivially. Let
$M^{\vee}=\Hom_{\e\Z\e}(M,\Z)$ be the linear dual of $M$.  Then $M^{\vee}$ has a natural  structure  of $J$-module by \cite[(1), p.~238]{ce} and
$M^{\vee}=(M/M_{\rm{tors}})^{\vee}$ is either zero or $\Z$-free. We have (see \cite[pp.~238-240]{ce})
\begin{equation}\label{eq1}
(M^{\vee})^{J}=(M_{J})^{\vee}
\end{equation}
and
\begin{equation}\label{eq2}
M^{\vee\vee}=M/M_{\rm{tors}}.
\end{equation}
In particular,
$M^{\vee\vee}=M$ if $M$ is $\Z$-free. The kernel of the canonical norm map
\begin{equation}\label{norm}
N\colon M\to M^{J},\quad m\mapsto\sum_{\sigma\in J}\sigma\e m,
\end{equation}
will be denoted by ${}_{N}M$. Now recall that the Tate cohomology groups
$\widehat{ H}^{r}(J,M)$, for $r\in\Z$, are defined as follows:
$\widehat{ H}^{r}(J,M)= H^{r}(J,M)$ if $r\geq 1$,
$\widehat{ H}^{r}(J,M)=H_{-r-1}(J,M)$ if  $r\leq -2$ and
\[
\begin{array}{rcl}
\widehat{H}^{\e 0}(J,M)&=&M^{J}/NM,\\
\widehat{H}^{-1}(J,M)&=&{}_{N}M/\frak A_{\be J}M.
\end{array}
\]
Note that, if $J$ acts trivially on $M$ and $M$ is $\Z$-free, then ${}_{N}M=0$ and therefore $\widehat{ H}^{-1}(J,M)=0$. Further, by \cite[Chapter XII, Proposition 2.5, p.~236, and
Exercise 3, p.~263]{ce}, $\widehat{ H}^{r}(J,M)$ is a finite group which is
annihilated by $|J\le|$ for every $r\in\Z$.

Next, if $M$ is a finitely generated abelian group, we will write
$M^{D}=\Hom_{\e\Z\e}(M,\Q/\Z)$. If $M$ is free, then  $M^{D}=M^{\vee}\otimes_{\Z}\e\Q/\Z$ by \cite[Chapter XII, beginning of \S3, pp.~237-238]{ce}.

Now assume that $M$ is a $\Z$-free and finitely generated $J$-module. Since $M^{D}=M^{\vee}\otimes_{\Z}\e\Q/\Z$ as noted above, the
short exact sequence $0\to\Z\to\Q\to \Q/\Z\to 0$ induces a short exact
sequence of $J$-modules $0\to M^{\vee}\to M^{\vee}\otimes_{\Z}\Q\to M^{D}\to
0$. The latter sequence induces, in turn, a canonical isomorphism of abelian groups
$\widehat{ H}^{\le r-1}(J,M^{D})=\widehat{ H}^{\le r}(J,M^{\vee})$ for every $r\in\Z$.
On the other hand, by \cite[Chapter XII, \S6, Theorem 6.4, p.~249]{ce},
$\widehat{ H}^{\le r-1}(J,M^{D})$ is canonically isomorphic to
$\widehat{ H}^{-r}(J,M)^{D}$. Thus, for every $r\in\Z$, there exists a
canonical isomorphism of finite abelian groups
\begin{equation}\label{dual}
\widehat{ H}^{r}(J,M)^{D}=\widehat{ H}^{-r}\be\big(J,M^{\vee}\big).
\end{equation}

\begin{lemma}\label{-1-coh} Let $M$ be a $\Z$-free and finitely generated
$J$-module.
\begin{enumerate}
\item[(i)] There exists a canonical exact sequence of abelian groups
\[
0\to  H^{1}(J,M)^{D}\to \big(M^{\vee}\big)_{\be J}\to
\big(M^{J}\e\big)^{\!\vee}\to 0.
\]
\item[(ii)] If $M^{J}\neq M$, then $M/M^{J}$ is $\Z$-free and $\big(M/M^{J}\e\big)^{\be J}=0$.
\end{enumerate}
\end{lemma}
\begin{proof}
Applying the snake lemma to the exact commutative diagram
\[
\xymatrix{0\ar[r]& \frak A_{\be J}\big(M^{\vee}\big)\ar[d]\ar[r]&
M^{\vee}\ar@{=}[d]\ar[r]& \big(M^{\vee}\big)_{\be J}\ar[d]\ar[r]& 0\\
0\ar[r]&{}_{N}\big(M^{\vee}\big)\ar[r]&M^{\vee}
\ar[r]&N\be\big(M^{\vee}\big)\ar[r]&0
}
\]
and using \eqref{dual}, we obtain an exact sequence
\begin{equation}\label{ex1}
0\to  H^{1}(J,M)^{D}\to
\big(M^{\vee}\big)_{\be J}\to N\be\big(M^{\vee}\big)\to 0.
\end{equation}
Now, since
$N\big(M^{\vee}\big)$ is a subgroup of the $\Z$-free group $M^{\vee}$ and
$ H^{1}(J,M)^{D}$ is torsion, the latter sequence induces, by \eqref{eq1} and \eqref{eq2}, canonical isomorphisms
\begin{equation}\label{exnm}
N\big(M^{\vee}\big)=\big(M^{\vee}\big)_{\be J}\big/\big(M^{\vee}\big)_{\be
J,\,\rm{tors}}=\big(\big(M^{\vee}\big)_{\be
J}\e\big)^{\vee\vee}=\Big(\big(M^{\vee\vee}\big)^{J}\e\Big)^{\vee}=
\big(M^{J}\e\big)^{\be\vee}.
\end{equation}
Assertion (i) is now clear. 

To prove (ii), we first note that, since $M^{J}$ is $\Z$-free, \eqref{exnm} yields a canonical isomorphism $N\big(M^{\vee}\big)^{\vee}=\big(M^{J}\e\big)^{\be\vee\vee}=M^{J}$. Therefore, taking the linear dual of the bottom row of the preceding diagram, we obtain an isomorphism
\begin{equation}\label{mmm}
({}_{N}(M^{\vee}))^{\vee}=M/M^{J}.
\end{equation}
Consequently, if $M^{J}\neq M$, then $M/M^{J}$ is $\Z$-free. Further, it follows from \eqref{ex1} that, if $M^{\prime}$ is a $\Z$-free and finitely generated $J$-module such that $NM^{\prime}=0$, then $M_{J}^{\prime}$ is torsion. Thus $({}_{N}(M^{\vee}))_{J}$ is torsion. Now, by \eqref{eq1} and \eqref{mmm}
\[
\big(M/M^{J}\e\big)^{\be J}=\big(({}_{N}(M^\vee))^{\vee}\le\big)^{\be J}= \big(({}_{N}(M^\vee))_J\big)^{\vee}=0,
\]
since $({}_{N}(M^{\vee}))_{J}$ is torsion.
\end{proof}

\subsection{A Canonical Resolution of $T$}
Let $L$ be the minimal splitting field of $T$, i.e., the fixed field of the kernel of the canonical homomorphism $\mathcal G\to\text{Aut}(X^{\lbe\ast}\be(T\e))$. We will write $T_{L}$ for the (split) $L$-torus $T\times_{\spec K}\spec L$. Let $J$ be the inertia subgroup of $\mathrm{Gal}(L/K)$. Then $I$ acts on the free and finitely generated $\Z$-module $X^{\lbe\ast}\be(T\e)$ through the finite quotient $J$ and \eqref{norm} is a map
\begin{equation}\label{norm2}
N\colon X^{\lbe\ast}\be(T\e)\to X^{\lbe\ast}\be(T\e)^{J}.
\end{equation}
Note that, since $\e  H^{1}(I^{\e\prime},X^{\lbe\ast}\be(T\e))=\Hom(I^{\e\prime},X^{\lbe\ast}\be(T\e))=0$ for any subgroup $I^{\e\prime}$ of $I$ which acts trivially on $X^{\lbe\ast}\be(T\e)$ (as $I^{\e\prime}$ is torsion and $X^{\lbe\ast}\be(T\e)$ is torsion-free), the inflation-restriction exact sequence
(see \cite[VII, \S6, Proposition 4, p.~117]{self}) shows that $ H^{1}(J,X^{\lbe\ast}\be(T\e))= H^{1}(I,X^{\lbe\ast}\be(T\e))$.

A $K$-torus $T$ is said to have {\it multiplicative reduction} if the special fiber
$\s T_{\be s}^{\e 0}$ of $\s T^{\e 0}$ is a $k$-torus. An equivalent condition is that $I$ act trivially on $X^{\lbe\ast}\be(T\e)$, i.e., $T$ splits over $\Knr$. If this is the case, then the
$\g$-module of characters of $\s T_{\be s}^{\e 0}$ is $X^{\lbe\ast}\be(T\e)$. See
\cite[Proposition-Definition 1.1, p.~462]{nx}. Further, there exists a canonical isomorphism of $\g$-modules
\begin{equation}\label{mris}
\phi(T\e)\overset{\!\sim}{\to} X_{\lbe\ast}\lbe(T\e).
\end{equation}
See \cite[Theorem 1.1.2, p.~29]{bra} and recall the identification 
$X^{*}\be(T)^{\vee}=X_{*}\lbe(T\e)$.

\begin{proposition} \label{phi-ex0} Let $0\to T_{1}\to T_{2}\to T_{3}\to 0$ be an
exact sequence of $K$-tori. Assume that the following conditions hold:
\begin{enumerate}
\item[(i)] ${\rm R}^{1}j_{*}T_{1}=0$ for the smooth topology on $S$, and
\item[(ii)] $\phi(T_{1})$ is torsion-free.
\end{enumerate}
Then the induced sequence of $\g$-modules $0\to \phi(T_{1})\to \phi(T_{2})\to
\phi(T_{3})\to 0$ is exact.
\end{proposition}
\begin{proof} By (i), the sequence of N\'eron models $0\to \s T_{1}\to \s T_{2}\to \s T_{3}\to 0$ is exact in the smooth topology. Thus, by \cite[Theorem 2.3.1, p.~50]{bra}, the induced sequence of $\g$-modules $\phi(T_{1})\to \phi(T_{2})\to \phi(T_{3})\to 0$ is exact. On the other hand, by \cite[Theorem 2.3.4, p.~52]{bra}, $\mathrm{Ker}\e[\e\phi(T_{1})\to \phi(T_{2})]$ is a finite $\g$-submodule of $\phi(T_{1})$, which is therefore zero by (ii). This completes the proof.
\end{proof}

\begin{corollary} \label{phi-ex} Let $0\to T_{1}\to T_{2}\to T_{3}\to 0$ be an
exact sequence of $K$-tori, where $T_{1}$ has {\it{multiplicative reduction}}.
Then the induced sequence of $\g$-modules $0\to \phi(T_{1})\to \phi(T_{2})\to
\phi(T_{3})\to 0$ is exact.
\end{corollary}
\begin{proof} Since $T_{1}$ splits over $\Knr$, ${\rm R}^{1}j_{*}T_{1}=0$ for the smooth topology on $S$ by \cite[Corollary 4.2.6, p.~82]{bra}. On the other hand, by \eqref{mris}, $\phi(T_{1})\simeq X_{*}(T_{1})$, which is torsion-free. The corollary is now immediate from the proposition.
\end{proof}

A $K$-torus $T$ is said to have {\it unipotent reduction} if the special fiber
$\s T_{\be s}^{\e 0}$ of $\s T^{\e 0}$ is a unipotent $k$-group scheme.

\begin{lemma}\label{unip-red} A $K$-torus $T$ has unipotent reduction if, and
only if, $X^{\lbe\ast}\be(T\e)^{I}=0$.
\end{lemma}
\begin{proof} By \cite[proof of Theorem 1.3]{nx}, $T$ has unipotent reduction
if, and only if, $T$ contains no nontrivial $K$-subtorus having multiplicative
reduction, i.e., $X^{\lbe\ast}\be(T\e)$ admits no free quotient on which $I$ (or, equivalently,
$J\e$) acts trivially. Assume that the latter holds and recall the norm map \eqref{norm2}. Since $NX^{\lbe\ast}\be(T\e)$ is a quotient of $X^{\lbe\ast}\be(T\e)$ with trivial $J$-action, we have $NX^{\lbe\ast}\be(T\e)=0$ and therefore $X^{\lbe\ast}\be(T\e)^{J}=X^{\lbe\ast}\be(T\e)^{J}/NX^{\lbe\ast}\be(T\e)=\widehat{ H}^{0}(J,X^{\lbe\ast}\be(T\e))$ is
a finite subgroup of the free group $X^{\lbe\ast}\be(T\e)$, i.e., $X^{\lbe\ast}\be(T\e)^{I}=X^{\lbe\ast}\be(T\e)^{J}=0$. Conversely,
assume that $X^{\lbe\ast}\be(T\e)^{I}=0$ and let $Y$ be an $I$-submodule of $X^{\lbe\ast}\be(T\e)$ such that $I$
acts trivially on $X^{\lbe\ast}\be(T\e)/Y$. Then the $I$-cohomology sequence associated to $0\to
Y\to X^{\lbe\ast}\be(T\e)\to X^{\lbe\ast}\be(T\e)/Y\to 0$ shows that $X^{\lbe\ast}\be(T\e)/Y$ is isomorphic to a subgroup of the finite
group $ H^{1}(I,Y)$. In particular, it is not free.
\end{proof}

What follows is an elaboration of \cite[Lemma 2.13]{xa}.

Let $T$ be any $K$-torus. The maximal quotient torus $T^{\le(m)}$ of $T$ having multiplicative reduction is
the $K$-torus with character module $X^{\lbe\ast}\be(T\e)^{I}$, and the maximal
subtorus $T_{(u)}$ of $T$ having unipotent reduction is the $K$-torus with
character module $X^{\lbe\ast}\be(T\e)/X^{\lbe\ast}\be(T\e)^{I}$. This follows from Lemmas \ref{-1-coh}(ii) and
\ref{unip-red} together with the fact that, if $Y$ is an $I$-submodule of $X^{\lbe\ast}\be(T\e)$ such that
$X^{\lbe\ast}\be(T\e)/Y$ is free and $(X^{\lbe\ast}\be(T\e)/Y\e)^{I}=0$, then $X^{\lbe\ast}\be(T\e)^{I}\subset Y$. Now the exact sequence of $\mathcal G$-modules $0\to X^{\lbe\ast}\be(T\e)^{I}\to X^{\lbe\ast}\be(T\e)\to X^{\lbe\ast}\be(T\e)/X^{\lbe\ast}\be(T\e)\textit{}^{I}\to 0$ induces an exact sequence of $K$-tori
\begin{equation}\label{umr}
0\to T_{(u)}\to T\to T^{(m)}\to 0.
\end{equation}

Recall now the minimal splitting field $L$ of $T$. The norm map
$N_{L/K}\colon L^{*}\to K^{*}$ induces an epimorphism
of $K$-tori
$R_{L/K}(T_{L})\to T$ whose kernel is denoted by $R^{(1)}_{L/K}(T_{L})$ and
called the {\it{norm one torus}} associated to $T$. See \cite[Theorem 0.4.4,
p.~16]{bra}. Thus there exist canonical exact sequences 
\begin{equation}\label{nt}
0\to R^{(1)}_{L/K}(T_{L})\to R_{L/K}(T_{L})\to T\to 0
\end{equation}
and
\begin{equation}\label{umr2}
0\to R^{(1)}_{L/K}(T_{L})_{(u)}\to R^{(1)}_{L/K}(T_{L})\to R^{(1)}_{L/K}(T_{L})^{(m)}\to 0,
\end{equation}
where the latter sequence is the sequence \eqref{umr} associated to the $K$-torus $R^{(1)}_{L/K}(T_{L})$.
Set
\begin{equation}\label{t1}
P=R^{(1)}_{L/K}(T_{L})^{(m)}
\end{equation}
and let $Q$ be the pushout of the canonical morphisms $R^{(1)}_{L/K}(T_{L})\hookrightarrow
R_{L/K}(T_{L})$ and $
R^{(1)}_{L/K}(T_{L})\twoheadrightarrow R^{(1)}_{L/K}(T_{L})^{(m)}=P$ appearing in  \eqref{nt} and \eqref{umr2}, respectively. Thus there exists a canonical exact commutative diagram
\[
\xymatrix{0\ar[r]&  R^{(1)}_{L/K}(T_{L})\ar@{->>}[d]\ar[r]&
R_{L/K}(T_{L})\ar@{->>}[d]\ar[r]& T\ar@{=}[d]\ar[r]& 0\\
0\ar[r]&P\ar[r]&Q\ar[r]&T\ar[r]& 0,
}
\]
where the top row is \eqref{nt}. By \eqref{umr2}, the kernel of the left-hand vertical map in the above diagram equals $R^{(1)}_{L/K}(T_{L})_{(u)}$, which immediately yields the formula
\begin{equation}\label{t2}
Q=R_{L/K}(T_{L})/R^{(1)}_{L/K}(T_{L})_{(u)}.
\end{equation}
Thus there exists a canonical exact sequence of $K$-tori
\begin{equation}\label{cres}
0\to P\to Q\to T\to 0,
\end{equation}
where $P$ and $Q$ are given by \eqref{t1} and \eqref{t2}, respectively. Since $P$ has multiplicative reduction, the following lemma is immediate from Corollary \ref{phi-ex}.
\begin{lemma}\label{cres1} The canonical resolution \eqref{cres} induces an exact sequence of $\g$-modules
\[
0\to \phi(P)\to \phi(Q)\to \phi(T\le)\to 0.
\]
\end{lemma}

\begin{lemma} \label{xseq} The canonical resolution \eqref{cres} induces an exact sequence of $\g$-modules \[
0\to X_{*}(P)\to X_{*}(Q)_{\lbe I}\to X_{*}(T)_{\lbe I}\to 0.\]
\end{lemma}
\begin{proof} The $J$-homology sequence associated to the short exact sequence of $J$-modules
$0\to X_{*}(P)\to X_{*}(Q)\to X_{*}(T)\to 0$ corresponding to \eqref{cres} is
\[
\dots\to  H_{1}(J, X_{*}(Q))\to   H_{1}\lbe\big(J,
X_{*}(T))\to X_{*}(P)_{\lbe I}\to
X_{*}(Q)_{\lbe I}\to X_{*}(T)_{\lbe I}\to 0.
\]
Since $ H_{1}(J, X_{*}(T))$ is torsion and $X_{*}(P)_{\lbe I}=X_{*}(P)$ is torsion-free, the lemma follows.
\end{proof}

\begin{lemma}\label{xq} $ H^{1}(I,X^{*}\be(Q))=0$.
\end{lemma}
\begin{proof} By \cite[Theorem 0.4.3, p.~14, and proof of Theorem 0.4.4, p.~16]{bra}, there exists a canonical isomorphism of $\mathcal G$-modules $X^{*}\be(R_{L/K}(T_{L}))=\Z^{d}[\mathrm{Gal}(L/K)]$, where $d$ is the dimension of $T$. Thus \eqref{t2} yields a canonical exact sequence of $\mathcal G$-modules
\[
0\to X^{*}\be(Q)\to  \Z^{d}[\mathrm{Gal}(L/K)]\to
X^{*}\be\big(R^{(1)}_{L/K}(T_{L})_{(u)}\big)\to 0.
\]
Now, since $X^{*}\be\big(R^{(1)}_{L/K}(T_{L})_{(u)}\big)^{\be J}=0$ by Lemma \ref{unip-red}, the $J$-cohomology sequence associated to the above short exact sequence yields an injection
\[
 H^{1}(I,X^{*}\be(Q))= H^{1}(J,X^{*}\be(Q))\e\hookrightarrow\e
 H^{1}(J,\Z^{d}[\mathrm{Gal}(L/K)]).
\]
Finally, since $\Z^{d}[\mathrm{Gal}(L/K)]$ is a free (right) $\Z^{d}[\e J\e]$-module of
finite rank, the latter cohomology group vanishes by Shapiro's lemma (see \cite[Lemma 6.3.2, p.~171]{we}), and this completes the proof.
\end{proof}

\section{Proof of Theorem \ref{phit}}\label{bit}

By Lemma \ref{-1-coh}(i), there exists a canonical exact sequence of $\g$-modules
\begin{equation}\label{iseq1}
0\longrightarrow   H^{1}(I,X^{\lbe\ast}\be(T\e))^{D}\longrightarrow  X_{\lbe\ast}\lbe(T\e)_{\lbe I}\overset{q_{\le T}}{\longrightarrow }
\big(X^{\lbe\ast}\be(T\e)^{I}\e\big)^{\!\vee}\longrightarrow  0.
\end{equation}
We will write $\underline{X}_{\le\ast}(T\e)_{\lbe I}$ and $\underline{X}^{*}\be(T\e)^{I}$, respectively, for the \'etale $k$-sheaves that correspond to the continuous $\g$-modules $X_{\lbe\ast}\lbe(T\e)_{\lbe I}$ and $X^{\lbe\ast}\be(T\e)^{I}$ (see \cite[Corollary II.2.2(i), p.~94]{t}).

\begin{lemma}\label{vtseq} There exists a canonical exact sequence of \'etale sheaves on $S$
\[
0\to\underline{\Hom}_{\e S}(\e j_{*}\underline{X}^{*}\be(T\e),\G_{m,S})\to   \s T\to i_{*}\big(\underline{X}^{*}\be(T\e)^{I}\e\big)^{\!\vee}\to 0.
\]
\end{lemma}
\begin{proof} Since $\underline{\Ext}^{1}_{\e S_{\mr{\acute{e}t}}}\!(\e j_{*}\underline{X}^{*}\be(T\e),\G_{m,S})=0$ by \cite[Theorem B.3, p.~131]{bra}, \eqref{pseq2} induces an exact sequence of \'etale sheaves on $S$
\[
0\to\underline{\Hom}_{\e S}(\e j_{*}\underline{X}^{*}\be(T\e),\G_{m,S})\to \underline{\Hom}_{\e S}(\e j_{*}\underline{X}^{*}\be(T\e),j_{*}\G_{m,K})\to \underline{\Hom}_{\e S}(\e j_{*}\underline{X}^{*}\be(T\e),i_{*}\Z_{\e k})\to 0.
\]
Now, since $\underline{X}^{*}\be(T\e)=j^{*}j_{*}\underline{X}^{*}\be(T\e)$ by \cite[Proposition II.8.1.1, p.~134]{t} and $i^{*}j_{*}\underline{X}^{*}\be(T\e)=\underline{X}^{*}\be(T\e)^{I}$ by \cite[Example II.3.12, p.~75]{mi1}, \eqref{tor} and \cite[Exercise II.3.22(a), p.~80]{mi1} yield canonical isomorphisms of \'etale sheaves on $S$
\begin{equation}\label{id1}
\underline{\Hom}_{\e S}(\e j_{*}\underline{X}^{*}\be(T\e),j_{*}\G_{m,K})=j_{*}\underline{\Hom}_{\e K}(\e j^{*}j_{*}\underline{X}^{*}\be(T\e),\G_{m,K})=j_{*}T=\s T
\end{equation}
and
\begin{equation}\label{id2}
\underline{\Hom}_{\e S}(\e j_{*}\underline{X}^{*}\be(T\e),i_{*}\Z_{\e k})=i_{*}\underline{\Hom}_{\e k}(i^{*}j_{*}\underline{X}^{*}\be(T\e), \e\Z_{\e k})=i_{*}\big(\underline{X}^{*}\be(T\e)^{I}\e\big)^{\!\vee}.
\end{equation}
The lemma is now clear.
\end{proof}

\begin{remark}\label{xrem} The same argument that proves  \eqref{id1} yields a canonical isomorphism of \'etale sheaves on $S$
\[
j_{*}\le\underline{X}_{*}(T\e)=
\underline{\Hom}_{\e S}(\e j_{*}\underline{X}^{*}\be(T\e),j_{*}\Z_{K}).
\]
\end{remark}

Let
\begin{equation}\label{vmor}
v_{\le T}\colon \s T\to i_{*}\big(\underline{X}^{*}\be(T\e)^{I}\e\big)^{\!\vee}
\end{equation}
be the epimorphism of \'etale sheaves which appears in the exact sequence of Lemma \ref{vtseq}. If $T=\G_{m,K}$, then $v_{\le T}=v_{\le \G_{m,K}}\colon j_{*}\G_{m,K}\to i_{*}\Z_{\e k}$ is the morphism appearing in the exact sequence \eqref{pseq2}. For arbitrary $T$, and via the identifications \eqref{id1} and \eqref{id2},
$v_{\le T}$ is the morphism
\[
\underline{\Hom}_{\e S}(\e j_{*}\underline{X}^{*}\be(T\e),v_{\le \G_{m,K}})\colon \underline{\Hom}_{\e S}(\e j_{*}\underline{X}^{*}\be(T\e),j_{*}\G_{m,K})\to \underline{\Hom}_{\e S}(\e j_{*}\underline{X}^{*}\be(T\e),i_{*}\Z_{\e k}).
\]

Now, since $\big(\underline{X}^{*}\be(T\e)^{I}\e\big)^{\!\vee}=i^{*}i_{*}\big(\underline{X}^{*}\be(T\e)^{I}\e\big)^{\!\vee}$ by \cite[Proposition II.8.1.1, p.~134]{t}, 
the exact sequence of Lemma \ref{vtseq} induces an exact sequence of \'etale $k$-sheaves
\[
0\longrightarrow i^{*}\underline{\Hom}_{\e S}(\e j_{*}\underline{X}^{*}\be(T\e),\G_{m,S})\longrightarrow i^{*}\s T\overset{i^{\be *}\be v_{T}}{\longrightarrow } \big(\underline{X}^{*}\lbe(T\e)^{I}\e\big)^{\!\vee}\longrightarrow 0.
\]
Since \eqref{ff} is an equivalence, $i^{*}\lbe v_{\le T}\colon i^{*}\s T\to \big(\underline{X}^{*}\be(T\e)^{I}\e\big)^{\!\vee}$ corresponds to a homomorphism of $k$-group schemes $\s T_{\lbe s}\to E$, where $E$ is the \'etale $k$-group scheme that represents $\big(\underline{X}^{*}\be(T\e)^{I}\e\big)^{\!\vee}$ (see \cite[Proposition II.9.2.3, p.~153]{t}). By diagram \eqref{diag1}, the latter homomorphism factors (uniquely) through a homomorphism of $k$-group schemes $\pi_{0}(\s T_{\lbe s})\to E$. Thus there exists a commutative diagram of \'etale $k$-sheaves
\begin{equation}\label{diag2}
\xymatrix{i^{*}\s T\ar@{->>}[r]\ar@{->>}[dr]_(0.4){i^{*}\! v_{\le T}}& \phi(T\e)\ar@{->>}[d]\\
&\big(\underline{X}^{*}\be(T\e)^{I}\e\big)^{\!\vee}.&
}
\end{equation}
Let
\begin{equation}\label{chom}
\alpha_{\e T}\colon \phi(T\e)\twoheadrightarrow \big(X^{\lbe\ast}\be(T\e)^{I}\e\big)^{\!\vee}
\end{equation}
be the epimorphism of $\g$-modules which corresponds to the vertical morphism in \eqref{diag2}. 

\begin{proposition} The map 
\[
\alpha_{\e T}^{\vee}\colon X^{\lbe\ast}\be(T\e)^{I}\hookrightarrow\phi(T\e)^{\vee}
\]
induced by \eqref{chom} is an isomorphism of $\g$-modules.
\end{proposition}
\begin{proof} See \cite[Theorem 5.1.6, p.~93]{bra} and note that the $\g$-module $E(T)=\mathrm{Coker}\e\, \alpha_{\e T}^{\vee}$ appearing there vanishes if $k$ is perfect by \cite[Theorem 5.3.8, p.~104]{bra}. 
\end{proof}

\begin{corollary}\label{iso2} The map \eqref{chom} induces an isomorphism of $\g$-modules 
\[
\phi(T\e)/\phi(T\e)_{\e\rm{tors}}\overset{\!\sim}{\to} \big(X^{\lbe\ast}\be(T\e)^{I}\e\big)^{\!\vee}.
\]
\end{corollary}
\begin{proof}
This is immediate from the proposition and \eqref{eq2}.
\end{proof}

\begin{lemma}\label{tors-free} If $\e  H^{1}(I,X^{\lbe\ast}\be(T\e))=0$, then there exists a canonical isomorphism of $\g$-modules
\[
\beta_{\e T}\colon \phi(T\e)\overset{\!\sim}{\to}X_{\lbe\ast}\lbe(T\e)_{\lbe I}\e.
\]
\end{lemma}
\begin{proof} By \cite[Proposition 2.7]{xa}, $\phi(T\e)$ is torsion free. Thus, by Corollary \eqref{iso2}, $\alpha_{\e T}\colon \phi(T\e)\to\big(X^{\lbe\ast}\be(T\e)^{I}\e\big)^{\!\vee}$ is an isomorphism of $\g$-modules. On the other hand, by \eqref{iseq1}, $q_{\e T}\colon X_{\lbe\ast}\lbe(T\e)_{\lbe I}\to \big(X^{\lbe\ast}\be(T\e)^{I}\e\big)^{\!\vee}$ is an isomorphism  of $\g$-modules. Thus $\beta_{\e T}:=q_{\e T}^{-1}\circ\alpha_{\e T}\colon \phi(T\e) \to X_{\lbe\ast}\lbe(T\e)_{\lbe I}$ is the required isomorphism of $\g$-modules. 
\end{proof}

We note that, if $T$ has multiplicative reduction (i.e., $I$ acts trivially on $X^{\lbe\ast}\be(T\e)$), then the isomorphism of the lemma is the isomorphism \eqref{mris}. Further, since the isomorphism of the previous lemma is {\it canonical} (i.e., functorial in $T\e$), given a morphism of $K$-tori $T_{1}\to T_{2}$ such that $ H^{1}(I,X^{\lbe\ast}\be(T_{1}\e))= H^{1}(I,X^{\lbe\ast}\be(T_{2}\e))=0$, the induced diagram
\begin{equation}\label{pt2}
\xymatrix{\phi(T_{1}\e)\ar[d]^{\beta_{\e T_{1}}}\ar[r]&
\phi(T_{2}\e)\ar[d]^{\beta_{\e T_{2}}}\\
X_{\lbe\ast}\lbe(T_{1}\e)_{\lbe I}\ar[r]&X_{\lbe\ast}\lbe(T_{2}\e)_{\lbe I}
}
\end{equation}
commutes.

Now recall the canonical resolution \eqref{cres} 
\[
0\to P\to Q\to T\to 0,
\]
where $P$ and $Q$ are given by \eqref{t1} and \eqref{t2}, respectively. Since $P$ has multiplicative reduction, we have $ H^{1}(I,X^{\lbe\ast}\be(P\e))=0$. Further, $ H^{1}(I,X^{\lbe\ast}\be(Q\e))=0$ by Lemma \ref{xq}. Thus, by \eqref{pt2} and Lemmas \ref{cres1}, \ref{xseq} and \ref{tors-free}, there exists a canonical exact commutative diagram
\begin{equation}\label{pt3}
\xymatrix{0\ar[r]& \phi(P\e)\ar[d]^(.45){\beta_{P}}_(.45){\simeq}\ar[r]&
\phi(Q\e)\ar[d]^(.45){\beta_{Q}}_(.45){\simeq}\ar[r]& \phi(T\e)\ar[r]& 0\\
0\ar[r]& X_{\lbe\ast}\lbe(P\e)_{\lbe I}\ar[r]&X_{\lbe\ast}\lbe(Q\e)_{\lbe I}\ar[r]&X_{\lbe\ast}\lbe(T\e)_{\lbe I}\ar[r]& 0.}
\end{equation}
It is now clear that there exists a unique isomorphism of $\g$-modules $\beta_{\e T}\colon \phi(T\e)\overset{\!\sim}{\to}X_{\lbe\ast}\lbe(T\e)_{\lbe I}$ such that the following diagram, derived from \eqref{pt3},
\[
\xymatrix{0\ar[r]& \phi(P\e)\ar[d]^(.45){\beta_{P}}_(.45){\simeq}\ar[r]&
\phi(Q\e)\ar[d]^(.45){\beta_{Q}}_(.45){\simeq}\ar[r]& \phi(T\e)\ar[d]^(.45){\beta_{T}}_(.45){\simeq}\ar[r]& 0\\
0\ar[r]& X_{\lbe\ast}\lbe(P\e)_{\lbe I}\ar[r]&X_{\lbe\ast}\lbe(Q\e)_{\lbe I}\ar[r]&X_{\lbe\ast}\lbe(T\e)_{\lbe I}\ar[r]& 0}
\]
commutes.

The isomorphism thus defined fits into a commutative diagram 
\[
\xymatrix{\phi(T\e)\ar[d]^(0.45){\beta_{T}}_(0.45){\simeq}\ar@{->>}[dr]^(0.45){\alpha_{ T}}&\\
X_{\lbe\ast}\lbe(T\e)_{\lbe I}\ar@{->>}[r]^(0.45){q_{T}}&\big(X^{\lbe\ast}\be(T\e)^{I}\e\big)^{\!\vee},
}
\]
where $q_{\le T}$ and $\alpha_{T}$ are the epimorphisms given by \eqref{iseq1} and \eqref{chom}, respectively.

The proof of Theorem \ref{phit} is now complete.

The following consequence of Theorem \ref{phit} was previously established in \cite[Corollary 2.18]{xa}.

\begin{corollary}\label{recov} There exists a canonical exact sequence of $\g$-modules
\[
0\to  H^{1}(I,X^{\lbe\ast}\be(T\e))^{D}\to\phi(T\e)\to\big(X^{\lbe\ast}\be(T\e)^{I}\e\big)^{\!\vee}\to 0.
\]
In particular, $\phi(T)_{\rm{tors}}$ is canonically isomorphic to $ H^{1}(I,X^{\lbe\ast}\be(T\e))^{D}$.
\end{corollary}
\begin{proof} This follows from Theorem \ref{phit} together with \eqref{iseq1}.
\end{proof}

\begin{remark}\label{fin} By Lemma \ref{unip-red} and the corollary, $T$ has unipotent reduction if, and only if, $\phi(T\e)$ is finite. If this is the case, then there exists a canonical isomorphism of finite $\g$-modules $\phi(T\e)= H^{1}(I,X^{\lbe\ast}\be(T\e))^{D}$.
\end{remark}

The following result clarifies the exactness properties of the functor $\phi(-)$.

\begin{proposition}\label{6-term} Let $0\to T_{1}\to T_{2}\to T_{3}\to 0$ be an exact sequence of $K$-tori. Then the given sequence of $K$-tori induces an exact sequence of $\g$-modules
\[
0\to  H^{2}(I, X^{\lbe *}\be(T_{1}\e))^{D}\to  H^{2}(I, X^{\lbe *}\be(T_{2}\e))^{D}\to   H^{2}(I,
X^{\lbe *}\be(T_{3}\e))^{D}\to\phi(T_{1})\to \phi(T_{2})\to \phi(T_{3})\to 0.
\]
\end{proposition}
\begin{proof} The exactness of the sequence $0\to  H^{2}(I, X^{\lbe *}\be(T_{1}\e))^{D}\to  H^{2}(I, X^{\lbe *}\be(T_{2}\e))^{D}\to   H^{2}(I,
X^{\lbe *}\be(T_{3}\e))^{D}$, which is induced by the short exact sequence of $I$-modules $0\to X^{\lbe *}\be(T_{3}\e)\to X^{\lbe *}\be(T_{2}\e)\to X^{\lbe *}\be(T_{1}\e)\to 0$, follows from the fact that $ H^{3}(I,
X^{\lbe *}\be(T_{3}\e))=0$ since $\Knr$ has cohomological dimension $\leq 1$ (see \cite[II, beginning of \S4.3, p.~85]{segc}). On the other hand, by the right exactness of the functor (from $\mathcal G$-modules to $\g$-modules) $M\mapsto M_{I}$ (see \cite[Exercise 6.1.1(2), p.~160]{we}), the short exact sequence
of $\mathcal G$-modules $0\to X_{\lbe *}\lbe(T_{1}\e)\to X_{\lbe *}\lbe(T_{2}\e)\to X_{\lbe *}\lbe(T_{3}\e)\to 0$ induces an exact sequence of $\g$-modules $X_{\lbe *}\lbe(T_{1}\e)_{\lbe I}\to X_{\lbe *}\lbe(T_{2}\e)_{\lbe I}\to X_{\lbe *}\lbe(T_{3}\e)_{\lbe I}\to 0$. By Theorem \ref{phit}, the latter sequence can be identified with a sequence $\phi(T_{1})\to \phi(T_{2})\to \phi(T_{3})\to 0$. Now the connecting homomorphism $ H^{2}(I,
X^{\lbe *}\be(T_{3}\e))^{D}\to\phi(T_{1})=X_{\lbe *}\lbe(T_{1}\e)_{\lbe I}$ in the sequence of the proposition factors as $ H^{2}(I,
X^{\lbe *}\be(T_{3}\e))^{D}\to  H^{1}(I,
X^{\lbe *}\be(T_{1}\e))^{D}\hookrightarrow X_{\lbe *}\lbe(T_{1}\e)_{\lbe I}$ (see \eqref{iseq1}), and its kernel is therefore equal to the kernel of $ H^{2}(I,
X^{\lbe *}\be(T_{3}\e))^{D}\to  H^{1}(I,
X^{\lbe *}\be(T_{1}\e))^{D}$, i.e., to the image of $ H^{2}(I,
X^{\lbe *}\be(T_{2}\e))^{D}\to   H^{2}(I,
X^{\lbe *}\be(T_{3}\e))^{D}$. It remains only to check exactness at $\phi(T_{1})$. By \cite[Theorem 2.3.4, p.~52]{bra}, the kernel of $\phi(T_{1})\to \phi(T_{2})$ is a finite $\g$-module, whence it agrees with
\[
\mathrm{Ker}\e[\e\phi(T_{1})_{\rm{tors}}\to \phi(T_{2})_{\rm{tors}}\e]=\mathrm{Ker}\e\be\big[\e  H^{1}(I, X^{\lbe *}\be(T_{1}\e))^{D}\to  H^{1}(I, X^{\lbe *}\be(T_{2}\e))^{D}\e\big]
\]
(see Corollary \ref{recov}). Since the latter group agrees with the image of $ H^{2}(I, X^{\lbe *}\be(T_{3}\e))^{D}\to  H^{1}(I, X^{\lbe *}\be(T_{1}\e))^{D}$, the proof is complete.
\end{proof}

The following corollary of the proposition generalizes Lemma \ref{phi-ex}.
\begin{corollary} Let $0\to T_{1}\to T_{2}\to T_{3}\to 0$ be an exact sequence of $K$-tori. If $\phi(T_{1}\e)$ is torsion-free, then the induced sequence of $\g$-modules
\[
0\to\phi(T_{1}\e)\to\phi(T_{2}\e)\to \phi(T_{3}\e)\to 0.
\]
is exact.
\end{corollary}
\begin{proof} As seen in the above proof, the connecting homomorphism $ H^{2}(I,
X^{\lbe *}\be(T_{3}\e))^{D}\to\phi(T_{1})$ in the exact sequence of the proposition factors through $ H^{1}(I,
X^{\lbe *}\be(T_{1}\e))^{D}=\phi(T_{1})_{\rm{tors}}=0$.
\end{proof}

\begin{remark}\label{et} Let $0\to T_{1}\to T_{2}\to T_{3}\to 0$ be as in the corollary, i.e., $\phi(T_{1}\e)$ is torsion-free. Further, for $i=1,2$ and $3$, let $\s T_{i}$ denote the N\'eron model of $T_{i}$ over $S$. By the corollary and \cite[Theorem II.2.15, p.~63]{mi1},  $0\to i_{*}\phi(T_{1}\e)\to i_{*} \phi(T_{2}\e)\to i_{*}\phi(T_{3}\e)\to 0$ is an exact sequence of \'etale sheaves on $S$. On the other hand, since ${\rm R}^{1}j_{*}T_{1}=0$ for the \'etale topology on $S$ (see \cite[Lemma 2.3]{xa}), $0\to\s T_{1}\to \s T_{2}\to \s T_{3}\to 0$ is an exact sequence of \'etale sheaves on $S$. Thus there exists a canonical exact commutative diagram of \'etale sheaves on $S$
\[
\xymatrix{0\ar[r]& \s T_{1}\ar@{->>}[d]\ar[r]&\s T_{2}
\ar@{->>}[d]\ar[r]& \s T_{3}\ar@{->>}[d]\ar[r]& 0\\
0\ar[r]&i_{*}\phi(T_{1}\e)\ar[r]&i_{*}\phi(T_{2}\e)\ar[r]&i_{*}\phi(T_{3}\e)\ar[r]& 0.
}
\]
By \eqref{pseq}, the above diagram yields an exact sequence $0\to\s T_{1}^{\e 0}\to \s T_{2}^{\e 0}\to \s T_{3}^{\e 0}\to 0$ of \'etale sheaves on $S$ and therefore an exact sequence of (representable) \'etale $k$-sheaves
\begin{equation}\label{sgps-0}
0\to i^{*}\!\s T_{1}^{\e 0}\to i^{*}\!\s T_{2}^{\e 0}\to i^{*}\!\s T_{3}^{\e 0}\to 0.
\end{equation}
Since \eqref{ff} is an equivalence, the latter sequence corresponds to a sequence of smooth, affine, commutative and connected algebraic $k$-group schemes
\begin{equation}\label{seq-gps}
0\to\s T_{1,s}^{\e 0}\overset{\be f}{\to} \s T_{2,s}^{\e 0}\overset{\be g}{\to}\s T_{3,s}^{\e 0}\to 0.
\end{equation}
The exactness of \eqref{sgps-0} implies that $g$ is surjective and $f$ identifies $\s T_{1,s}^{\e 0}$ with $\mathrm{Ker}\e g:=\s T_{2,s}^{\e 0}\times_{\s T_{3,s}^{\e 0}}\spec k$. See, for example, \cite[Lemma 2.1]{bg1}. Now \cite[Lemma 2.5]{bg1} shows that the sequence of representable presheaves on $\spec k$ induced by \eqref{seq-gps} is an exact sequence of fppf and fpqc sheaves on $\spec k$. In other words, \eqref{seq-gps} is exact for the \'etale, fppf and fpqc topologies on $\spec k$.
\end{remark}

\section{The Cohomology of Tori}\label{t-coh}

All cohomology groups below are taken with respect to the \'etale topology on the relevant scheme.

Let $T$ be a $K$-torus. The maximal subtorus $T_{(m)}$ of $T$ having multiplicative reduction is the $K$-torus with character module $NX^{\lbe\ast}\be(T\e)$ (see \cite[Proposition 1.2]{nx}). On the other hand, the maximal quotient torus $T^{(u)}$ of $T$ having unipotent reduction
is the $K$-torus with character module ${}_{N}X^{\lbe\ast}\be(T\e)$. Indeed, if $Y$ is a
$\mathcal G$-submodule of $X^{\lbe\ast}\be(T\e)$ such that $Y^{I}=Y^{J}=0$, then $NY=0$, i.e., $Y\subset
{}_{N}X^{\lbe\ast}\be(T\e)$. Now the exact sequence of $\mathcal G$-modules $0\to {}_{N}X^{\lbe\ast}\be(T\e)\to X^{\lbe\ast}\be(T\e)\to NX^{\lbe\ast}\be(T\e)\to 0$
induces an exact sequence of $K$-tori 
\begin{equation}\label{tseq}
0\to T_{(m)}\to T\to T^{(u)}\to 0.
\end{equation}
Let 
$\s T_{(m)}$, $\s T$ and $\s T^{(u)}$ denote, respectively, the N\'eron models of  $T_{(m)},T$ and $T^{(u)}$ over $S$. Since $X^{\lbe\ast}\lbe(T_{(m)})=NX^{\lbe\ast}\be(T\e)$ is torsion-free, Remark \ref{et} yields a sequence of
smooth, affine, commutative and connected algebraic $k$-groups
\begin{equation}\label{too}
0\to\s T_{(m),\le s}^{\e 0}\to\s  T^{\e 0}_{s}\to
\big(\s T^{(u)}\big)_{\lbe s}^{\be 0}\to 0.
\end{equation}
The latter sequence is exact for the \'etale, fppf and fpqc topologies on $\spec k$. Set
\begin{equation}\label{tau}
\tau=\s T_{(m),\le s}^{\e 0}\,,
\end{equation}
which is the unique maximal $k$-torus of $\s  T^{\e 0}_{s}$. Note that the $\g$-module of characters of $\tau$ is $NX^{\lbe\ast}\be(T\e)$. In particular, if $T$ has multiplicative
reduction, so that $J=1$ and $NX^{\lbe\ast}\be(T\e)=X^{\lbe\ast}\be(T\e)$, then the character module of $\tau$ is
$X^{\lbe\ast}\be(T\e)$ regarded as a $\g$-module.

\begin{lemma}\label{tau-coh} For every $r\geq 1$, there exists a canonical
isomorphism of abelian groups
$ H^{r}(k,\s  T^{\e 0}_{s})= H^{r}(k,\tau)$.
\end{lemma}
\begin{proof} Since $k$ is perfect, the sequence \eqref{too} splits
(see \cite[XVII, Theorem 6.1.1]{sga3}), i.e., there exists an isomorphism of $k$-group schemes $\s T_{\lbe s}^{\e 0}\simeq\tau\times U$, where
$U=\big(\s T^{(u)}\big)_{\lbe s}^{\be 0}$. Since $U$ is algebraic, smooth,
connected and unipotent, it has a composition series whose successive quotients
are $k$-isomorphic to $\G_{a,k}$ (see \cite[XVII, Corollary 4.1.3]{sga3}). Thus,
since $ H^{r}(k, \G_{a})=0$ for every $r\geq 1$ by \cite[Chapter X, \S1,
Proposition 1, p.~150]{self}, we have $ H^{r}(k, U)=0$ for every $r\geq 1$. The
lemma is now clear.
\end{proof}

Now, since ${\rm R}^{\lbe s}j_{*}T=0$ for the \'etale topology on $S$ for
all $s>0$ by \cite[Lemma 2.3]{xa}, the Leray spectral sequence in \'etale
cohomology
$ H^{\le r}(S,{\rm R}^{\lbe s}\be j_{*}T\e)\implies  H^{\le r+s}(K,T\e)$ yields isomorphisms
$ H^{\le r}(S,\s T\e)= H^{\le r}(S,j_{*}T\e)= H^{r}(K,T\e)$ for every $r\geq 0$. On the other hand, $ H^{\le r}(S,\s T^{\e 0})= H^{\le r}(k,\s T^{\e 0}_{\lbe s})= H^{\le r}(k,\tau)$ for every $r\geq 1$ by Lemma \ref{tau-coh} and \cite[Theorem 11.7]{dix}. Further,
$ H^{\le r}(S,i_{*}\phi(T\e))= H^{\le r}(k,i^{*}i_{*}\phi(T\e))= H^{\le r}(k,\phi(T\e))$ for every $r\geq 0$ by \cite[Proposition II.1.1(b), p.149]{adt}. Thus, by \eqref{pseq}, there exists a canonical exact sequence of abelian groups
\begin{equation}\label{long-coh}
\dots\to  H^{\le r}(k,\tau)\to  H^{\le r}(K,T\e)\to
 H^{\le r}(k,\phi(T\e))\to \dots
\end{equation}
where $r\geq 1$ and $\tau$ is the $k$-torus \eqref{tau}. By Theorem \ref{phit}, the preceding sequence is canonically isomorphic to a sequence
\begin{equation}\label{long-coh1}
\dots\to  H^{\le r}(k,\tau)\to  H^{\le r}(K,T\e)\to
 H^{\le r}(k,X_{\lbe\ast}\lbe(T\e)_{\lbe I})\to \dots.
\end{equation}
We now derive some consequences of \eqref{long-coh} and \eqref{long-coh1}.

\begin{proposition}\label{u-r-t} Assume that $T$ has unipotent reduction.
Then, for every $r\geq 1$, there exists a canonical isomorphism of abelian groups
\[
 H^{\le r}(K,T\e)= H^{\le r}\be\big(k, H^{\le 1}(I,X^{\lbe\ast}\be(T\e))^{D}\le\big).
\]
\end{proposition}
\begin{proof} This is clear from \eqref{long-coh} and Remark \ref{fin} since $\tau=0$ in this
case.
\end{proof}

The next proposition generalizes \cite[Example III.2.22(c), p.~108]{mi1}
(at least when the ring $A$ appearing in [loc.cit.] is complete).

\begin{proposition} \label{mult-red}
Assume that $T$ has multiplicative reduction. Then, for every $r\geq 1$, the
sequence of abelian groups induced by \eqref{long-coh1}
\[
0\to  H^{\le r}(k,\tau)\to  H^{\le r}(K,T\e)\to  H^{\le r}(k,X_{\lbe\ast}\lbe(T\e))\to 0
\]
is split exact.
\end{proposition}
\begin{proof} Since $I$ acts trivially on $X_{*}\lbe(T\e)$, \eqref{vmor} is a morphism $v_{\le T}\colon \s T\to i_{*}\le\big(\underline{X}^{*}\be(T\e)^{I}\e\big)^{\!\vee}=i_{*}\underline{X}_{\le *}\be(T\e)$ and the map $ H^{r}(K,T\e)\to  H^{r}(k,X_{\lbe\ast}\lbe(T\e))$ appearing in the sequence of the proposition can be identified with the homomorphism $ H^{r}(v_{\le T})\colon  H^{r}(S,\s T\e)\to  H^{r}(S,i_{*}\le\underline{X}_{\le *}\be(T\e))$ induced by $v_{\le T}$. Recall also that, via the identifications \eqref{id1} and \eqref{id2}, $v_{\le T}$ can be identified with the morphism
\[
\underline{\Hom}_{\e S}(\e j_{*}\underline{X}^{*}\be(T\e),v_{\le \G_{m,K}})\colon \underline{\Hom}_{\e S}(\e j_{*}\underline{X}^{*}\be(T\e),j_{*}\G_{m,K})\to \underline{\Hom}_{\e S}(\e j_{*}\underline{X}^{*}\be(T\e),i_{*}\Z_{\e k}),
\]
where $v_{\e \G_{m,K}}\colon j_{*}\G_{m,K}\to i_{*}\Z_{\e k}$ is the morphism appearing in the exact sequence \eqref{pseq2}. Now choose a uniformizer $\pi\in A$, let $u\colon\Z_{K}\to\G_{m,K}$ be the homomorphism of $K$-group schemes which maps $1\in\Z$ to $\pi\in \G_{m,K}\lbe(K)=K^{*}$ and let $u_{\e \G_{m,K}}\colon j_{*}\le\Z_{K}\to j_{*}\G_{m,K}$ be the morphism of \'etale sheaves on $S$ induced by $u$. By Remark \ref{xrem}, there exists a canonical isomorphism (of \'etale sheaves on $S$) $j_{*}\le\underline{X}_{*}\lbe(T\e)=
\underline{\Hom}_{\e S}(\e j_{*}\underline{X}^{*}\be(T\e),j_{*}\Z_{K})$. Thus 
\[
\underline{\Hom}_{\e S}(\e j_{*}\underline{X}^{*}\be(T\e),u_{\e \G_{m,K}})\colon \underline{\Hom}_{S}(j_{*}\underline{X}^{*}\be(T\e),j_{*}\Z_{K})\to \underline{\Hom}_{S}(j_{*}\underline{X}^{*}\be(T\e),j_{*}\G_{m,K})
\]
can be identified with a morphism
$u_{\le T}\colon j_{*}\le\underline{X}_{\le *}\lbe(T\e)\to \s T$. Clearly, the morphism of \'etale sheaves on $S$
\[
v_{\le T}\circ u_{\le T}\colon j_{*}\underline{X}_{\le *}\lbe(T\e)\to i_{*}\le\underline{X}_{\le *}\lbe(T\e)
\]
can be identified with $\underline{\Hom}_{\e S}(\e j_{*}\underline{X}^{*}\be(T\e),v_{\le \G_{m,K}}\!\be\circ\be  u_{\e \G_{m,K}})$, where $v_{\le \G_{m,K}}\!\circ u_{\e \G_{m,K}}\colon j_{*}\Z_{K}\to i_{*}\Z_{k}$.
On the other hand, by \cite[Proposition II.8.2.1, p.~142]{t}, for any \'etale sheaf $\s F$ on $S$ there exists a canonical exact exact sequence of \'etale sheaves on $S$
\[
0\to j_{!}j^{*}\!\s F\to \s F\to i_{*}i^{*}\!\s F\to 0.
\]
Setting $\s F=j_{*}\underline{X}_{\le *}\lbe(T\e)$ above and noting that $j^{*}j_{*}\underline{X}_{\le *}\lbe(T\e)=\underline{X}_{\le *}\lbe(T\e)$ and $i^{*}j_{*}\underline{X}_{\le *}\lbe(T\e)=\underline{X}_{\le *}\lbe(T\e)$, we obtain a canonical exact sequence of \'etale sheaves on $S$
\[
0\longrightarrow  j_{!}\le\underline{X}_{\le *}\lbe(T\e)\longrightarrow  j_{*}\underline{X}_{\le *}\lbe(T\e)\stackrel{\! w_{T}}{\longrightarrow }i_{*}\le\underline{X}_{\le *}\lbe(T\e)\longrightarrow  0.
\]
It is immediate that $v_{\le \G_{m,K}}\be\circ u_{\e \G_{m,K}}=w_{\e \G_{m,K}}$, which implies that $v_{\le T}\circ u_{\le T}=w_{\le T}$ for any $T$. On the other hand, since $ H^{r}(S,j_{!}\le\underline{X}_{\le *}\lbe(T\e))=0$ for every $r\geq 1$ by \cite[Proposition II.1.1(a), p.~149]{adt}, the map $ H^{r}(w_{\le T})\colon  H^{r}(S,j_{*}\underline{X}_{\le *}\lbe(T\e)\e)\to  H^{r}(S,i_{*}\le\underline{X}_{\le *}\lbe(T\e))$ induced by $w_{\le T}$ is an isomorphism for every $r\geq 1$ and we conclude that $ H^{r}(u_{\le T})\circ  H^{r}(w_{\le T})^{-1}$ is a section (i.e., right inverse) of $ H^{r}(v_{\le T})$.
\end{proof}

The following lemma is well-known but we were unable to find an appropriate reference.
 
\begin{lemma}\label{coh-triv} Assume that $k$ has cohomological dimension $\leq 1$. If $\tau$ is a $k$-torus, then $ H^{r}(k,\tau\e)=0$ for every $r\geq 1$.
\end{lemma}
\begin{proof} Since $k$ is perfect, $k$ is a field of dimension $\leq 1$ by \cite[Chapter II, \S3.1, Proposition 6(b), p.~78]{segc}. Thus, by \cite[Chapter II, \S3.1, Proposition 5(iii), p.~78]{segc}, for any finite Galois extension $l/k$, $l^{*}$ is a cohomologically trivial $\text{Gal}(l/k)$-module. Let $l$ be the minimal splitting field of $\tau$. Then, since $X^{*}\be(\tau)$ is a free (and therefore projective) $\Z$-module, we have $\text{Ext}^{1}_{\Z}(X^{*}\be(\tau),l^{*})=0$, whence $\tau(l)=\Hom(X^{*}\be(\tau),l^{*})$ (see \eqref{thom}) is cohomologically trivial as well by \cite[Chapter IX, \S5, Theorem 9, p.~145]{self}. Finally, since $ H^{r}(k,\tau\e)$ is the inductive limit of the groups $ H^{r}(\text{Gal}(l/k),\tau(l)\e)$ as $l/k$ ranges over the set of all finite Galois subextensions of $\kbar/k$, the proof is complete.
\end{proof}

\begin{theorem} \label{kt2} Assume that $k$ has cohomological dimension $\leq 1$. Then
\begin{enumerate}
\item[(i)] The sequence $0\to\s T^{\e 0}(\lbe A)\to T(K)\to \phi(T\e)(k)\to 0$ is exact.
\item[(ii)] For $r=1$ and $2$, there exist canonical isomorphisms
\[
 H^{\le r}(K,T\e)\simeq  H^{\le r}\be\big(k,X_{\lbe\ast}\lbe(T\e)_{\lbe I}\lbe\big)
\]
\end{enumerate}
If $r\geq 3$, the groups $ H^{r}(K,T\e)$ vanish.
\end{theorem}
\begin{proof} The last assertion follows from \eqref{long-coh} since
$ H^{r}(k,\tau)= H^{r}(k,\phi(T\e))=0$ for $r\geq 3$. Assertions (i) and (ii) are immediate from \eqref{long-coh}, \eqref{long-coh1} and Lemma \ref{coh-triv}. 
\end{proof}

\begin{remark} If $k$ has cohomological dimension $\leq 1$, then $ H^{r}\be\big(k, H^{1}(I,X^{\lbe\ast}\be(T\e))^{D}\le\big)=0$ for $r\geq 2$ since $ H^{1}(I,X^{\lbe\ast}\be(T\e))^{D}$ is a finite $\g$-module. Thus, by \eqref{iseq1} and assertion (ii) of the theorem,
\[
 H^{2}(K,T\e)= H^{2}\big(k,X_{\lbe\ast}\lbe(T\e)_{\lbe I}\be\big)= H^{2}\big(k,(X^{\lbe\ast}\be(T\e)^{I})^{\vee}\le\big).
\]
\end{remark}

The preceding remark is a particular case of the following proposition.

\begin{proposition}\label{fprop} Assume that $k$ has finite cohomological dimension $n\geq 1$. Then there exists a canonical isomorphism of divisible abelian groups
\[
 H^{\e n+1}(K,T\e)= H^{\e n+1}\be\big(k,(X^{\lbe\ast}\be(T\e)^{I})^{\vee}\le\big).
\]
If $r\geq n+2$, the groups $ H^{r}(K,T\e)$ vanish.
\end{proposition}
\begin{proof} The group on the right above is divisible by \cite[Corollary 1, p.~55]{sh1}. Now, since $\tau\lbe\big(\e\kbar\e\big)$ is divisible, we have $ H^{\le r}(k,\tau)=0$ for every $r\geq n+1$ by \cite[Proposition 14, p.~54]{sh1}. Thus \eqref{long-coh1} yields a canonical isomorphism of abelian groups $ H^{\le r}(K,T\e)= H^{\le r}(k,X_{\lbe\ast}\lbe(T\e)_{\lbe I})$ for every $r\geq n+1$. On the other hand, since $ H^{1}(I,X^{\lbe\ast}\be(T\e))^{D}$ is finite, we have $ H^{ r}\be\big(k, H^{1}(I,X^{\lbe\ast}\be(T\e))^{D}\le\big)=0$ for all $r\geq n+1$ and \eqref{iseq1} yields isomorphisms $ H^{\le r}(k,X_{\lbe\ast}\lbe(T\e)_{\lbe I})= H^{\le r}\big(k,(X^{\lbe\ast}\be(T\e)^{I})^{\vee}\e\big)$ for each $r\geq n+1$. The latter group vanishes if $r\geq n+2$ by \cite[Proposition 14, p.~54]{sh1}, whence the proposition follows.
\end{proof}

\section{Abelian Cohomology of Reductive Groups}\label{ab-coh}
Assume that $k$ has cohomological dimension $\leq 1$.
Recall that a $K$-torus $F$ is called {\it flasque} if the $\mathcal G$-module $X_{*}\lbe(F\e)$ is $ H^{1}$-trivial (see \cite[Lemma 1(iv), p.~179]{cts}). By \cite[Proposition-Definition 3.1, p.~88, and Proposition 2.2, p.~86]{ct}, any connected reductive algebraic group $G$ over $K$ admits a {\it flasque resolution}, i.e., there exists a central extension 
\begin{equation}\label{f-res}
1\to F\to H\to G\to 1,
\end{equation}
where $F$ is a flasque $K$-torus and $H$ is a connected reductive algebraic over $K$ such that the derived group $H^{{\rm{der}}}$ of $H$ is simply connected and $R:=H^{\tor}:=H/H^{{\rm{der}}}$ is a quasi-trivial $K$-torus. The finitely generated  (continuous) $\mathcal G$-module $\pi_{1}(G):=\mathrm{Coker}\e [X_{*}\lbe(F\e)\to X_{*}\lbe(R\e)\e]$ is independent (up to isomorphism) of the choice of resolution \eqref{f-res} and is called the {\it algebraic fundamental group of $G$}. See \cite[Proposition-Definition 6.1, p.~102]{ct}. Recall from \cite[Corollary 4.3]{ga1} that, if $r\geq 1$ is an integer, the {\it $r$-th (flat) abelian cohomology group of $G$} may be defined as the flat hypercohomology group
\[
 H^{\le r}_{\rm{ab}}(K_{\rm{fl}}, G \e):={\mathbb H}^{\e r}\lbe
(K_{\lbe{\rm fppf}},\pi_{1}\lbe(G\e)\!\otimes^{\e\mathbf{L}}\!
\G_{m,K}).
\]

These abelian groups are of interest because they can be related to the pointed Galois cohomology sets
$ H^{r}(K,G)$ for $r=1$ and $2$. See Corollary \ref{abcor} below for the case $r=1$.

Now, since $R$ is quasi-trivial, we have $ H^{1}(K,R)=0$. Further, $ H^{r}(K,F)=0$ for every $r\geq 3$ by the last assertion of Theorem \ref{kt2}. Thus, by \cite[Proposition 4.2]{ga1}, $ H^{\le r}_{\rm{ab}}(K_{\rm{fl}},G\e)=0$ for every $r\geq 3$ and \eqref{f-res} induces an exact sequence of abelian groups
\begin{equation}\label{ab1}
0\to  H^{\le 1}_{\rm{ab}}(K_{\rm{fl}},G\e)\to  H^{\le
2}(K,F\e)\to  H^{\le 2}(K,R\e)\to  H^{\le 2}_{\rm{ab}}(K_{\rm{fl}},G\e)\to 0.
\end{equation}
Now, by \cite[Proposition 6.2, p.~102]{ct}, there exists a canonical exact sequence of $\mathcal G$-modules
\begin{equation}\label{piseq}
0\to X_{*}\lbe(F\e)\to X_{*}\lbe(R\e)\to\pi_{1}(G)\to 0
\end{equation}
which induces a short exact sequence of $\g$-modules
\begin{equation}\label{piI}
0\to X_{*}\lbe(F\e)_{\lbe I}/M\to X_{*}\lbe(R\e)_{\lbe I}\to\pi_{1}(G)_{\lbe I}\to 0,
\end{equation}
where $M$ is a finite submodule of $X_{*}\lbe(F\e)_{\lbe I}$ which is isomorphic to a quotient of $H_{1}\be\big(J,\pi_{1}(G)\e\big)$ (see the proof of Lemma \ref{xseq}). 
Since $ H^{\le r}(k,M)=0$ for every $r\geq 2$, we have $ H^{\le 2}(k,X_{*}\lbe(F\e)_{\lbe I}/M)= H^{\le 2}(k,X_{*}\lbe(F\e)_{\lbe I}\lbe)$ and $ H^{\le r}(k,X_{*}\lbe(F\e)_{\lbe I}/M)=0$ for $r\geq 3$ (see \cite[Proposition 14, p.~54]{sh1}). Further, $ H^{\le 1}(k,X_{*}\lbe(R\e)_{\lbe I})= H^{1}(K,R)=0$ by Theorem \ref{kt2}(ii). Thus \eqref{piI} induces an exact sequence of abelian groups
\begin{equation}\label{ab2}
0\to  H^{\le 1}(k,\pi_{1}(G)_{\lbe I})\to  H^{\le
2}(k,X_{*}\lbe(F\e)_{\lbe I}\e)\to  H^{\le 2}(k,X_{*}\lbe(R\e)_{\lbe I})\to  H^{\le 2}(k,\pi_{1}(G)_{\lbe I})\to 0.
\end{equation}
Now consider the exact commutative diagram
\[
\xymatrix{0\ar[r]& H^{\le 1}_{\rm{ab}}(K_{\rm{fl}},G\e)\ar[r]&
 H^{\le
2}(K,F\e)\ar[d]^{\sim}\ar[r]&  H^{\le
2}(K,R\e)\ar[d]^{\sim}\ar[r]&  H^{\le 2}_{\rm{ab}}(K_{\rm{fl}},G\e)\ar[r]&0\\
0\ar[r]&  H^{\le 1}(k,\pi_{1}(G)_{\lbe I})\ar[r]&
 H^{\le 2}(k,X_{*}\lbe(F\e)_{\lbe I}\e)\ar[r]&  H^{\le
2}(k,X_{*}\lbe(R\e)_{\lbe I}\e)\ar[r]&  H^{\le 2}(k,\pi_{1}(G)_{\lbe I})\ar[r]&0
}
\]
whose top and bottom rows are the sequences \eqref{ab1} and \eqref{ab2}, respectively, and vertical maps are the isomorphisms of Theorem \ref{kt2}(ii). Clearly, the maps $ H^{\le
2}(K,F\e)\stackrel{\be\sim}{\to} H^{\le 2}(k,X_{*}\lbe(F\e)_{\lbe I})$ and $ H^{\le
2}(K,R\e)\stackrel{\be\sim}{\to} H^{\le 2}(k,X_{*}\lbe(R\e)_{\lbe I})$ induce isomorphisms
\[
H^{r}_{\rm{ab}}(K_{\rm{fl}},G\le)\stackrel{\sim}{\to} H^{\le r}(k,\pi_{1}(G)_{\lbe I}\lbe)
\]
for $r=1$ and $2$. Thus the following holds.

\begin{theorem}\label{abt} Assume that $k$ has cohomological dimension $\leq 1$ and let $G$ be a connected reductive algebraic group over $K$. Then the flasque resolution \eqref{f-res} induces isomorphisms of abelian groups
\[
H^{r}_{\rm{ab}}(K_{\rm{fl}},G\e)\simeq  H^{r}(k,\pi_{1}(G\le)_{\lbe I}\lbe)
\]
for $r=1$ and $2$. If $r\geq 3$, the groups $ H^{\le r}_{\rm{ab}}(K_{\rm{fl}},G\e)$ vanish.
\end{theorem}

\begin{corollary}\label{abcor}  Assume that $k$ has cohomological dimension $\leq 1$. Then there exists a bijection of pointed sets
\[
H^{1}(K,G\le)\simeq H^{\le 1}(k,\pi_{1}(G\le)_{\lbe I}\lbe).
\]
In particular, $H^{1}(K,G\le)$ can be endowed with an abelian group structure.
\end{corollary}
\begin{proof} Let $\widetilde{G}$ be the simply connected central cover of $G^{\der}$ (see \cite[p.~1161]{ga1} for the definition of $\widetilde{G}\,$). By \cite[Theorem 4.7(ii), p.~697, and Remark 3.16(3), p.~695]{bt}, the Galois cohomology set
$H^{1}\big(K,\widetilde{G}\e\big)$ is trivial. Further, by \cite[VII, Theorem 3.1, p.~99]{do}, $K$ is a field of Douai type in the sense of \cite[Definition 5.2]{ga2}. Thus, by \cite[Theorem 5.8(i)]{ga2}, the first abelianization map ${\rm{ab}}^{1}\colon H^{1}(K,G\le)\to H^{\le 1}_{\rm{ab}}(K_{\rm{fl}},G\le)$ is bijective. The corollary is now immediate from the theorem. 
\end{proof}

\begin{corollary} \label{abcor2} Assume that $k$ is quasi-finite. Then there exists a bijection of pointed sets $H^{1}(K,G\le)\simeq\pi_{1}(G)_{\mathcal G, {\rm{tors}}}$.
\end{corollary}
\begin{proof} Since a quasi-finite field is perfect and of cohomological dimension $\leq 1$ by \cite[XIII, beginning of \S2, p.~190]{self} and \cite[III, \S2, Corollary 3, p.~69]{sh1}, the corollary is immediate from the previous corollary and the following lemma.
\end{proof}

\begin{lemma}\label{bor} Assume that $k$ is quasi-finite and let $M$ be a continuous $\g$-module. Then there exists a canonical isomorphism of abelian groups
$H^{\le 1}(k,M\lbe)\simeq M_{\g\le,\e\rm{tors}}$.
\end{lemma}
\begin{proof} Let $\sigma$ be a free generator of $\g$.
By \cite[XIII, \S 1, Proposition 1, p.~189]{self}, there exists a canonical 
isomorphism of abelian groups $H^{\le 1}(k,M\lbe)=M^{\e\prime}/(\sigma-1)M$, where $M^{\e\prime}$ is the subgroup of $M$ consisting of those $x\in M$ for which there exists an integer $n\geq 1$ such that $(1+\sigma+\cdots \sigma^{\le n-1})\le x=0$. Thus it remains only to check that $M^{\e\prime}/(\sigma-1)M$ is the full torsion subgroup of $M_{\g}=M/(\sigma-1)M$. Let $x\in M$ be such that $mx=(\sigma-1)y$ for some positive integer $m$ and some $y\in M$ and choose a positive integer $r$ such that $\sigma^{\le r}x=x$ and $\sigma^{\le r}y=y$. Then
\[
\begin{array}{rcl}
(1+\sigma+\cdots \sigma^{\le mr-1})\le x&=&(1+\sigma+\cdots \sigma^{\le r-1})(1+\sigma^{\le r}+\cdots \sigma^{\le (m-1)r})x\\
&=&(1+\sigma+\cdots \sigma^{\le r-1})mx=(1+\sigma+\cdots \sigma^{\le r-1})(\sigma-1)y\\
&=&(\sigma^{\le r}-1)y=0.
\end{array}
\]
\end{proof}

\begin{remarks}\indent
\begin{enumerate}
\item[(a)] When $K$ is a finite extension of $\Q_{\,p}$, Corollary \ref{abcor2} is due to  Borovoi. See \cite[Corollary 5.5(i)]{bor}.
\item[(b)] Assume that $k$ has finite cohomological dimension $n\geq 1$ and let $\mu$ be the  kernel of the canonical morphism $\widetilde{G}\to G$. Then $\mu$ is a finite and commutative $K$-group scheme. By \cite[~\S II.4.3, Proposition 12, p.~85]{segc} and \cite[Theorem 4, p.~593]{sh2}, $ H^{r}\be(K_{\rm{fl}},\mu)=0$ for every $r\geq n+2$. Consequently, the exact sequence in \cite[p.~1174, line 8]{ga1} yields an isomorphism of abelian groups $ H^{\le n+1}_{\rm{ab}}(K_{\rm{fl}},G\e)\simeq  H^{\le n+1}(K,G^{\tor})$. Thus, by the proof of Proposition \ref{fprop}, there exists an isomorphism 
$ H^{\le n+1}_{\rm{ab}}(K_{\rm{fl}},G\e)\simeq  H^{\le n+1}(k,X_{*}(G^{\tor})_{\lbe I})$. On the other hand, it follows from \cite[Proposition 6.4, p.~104]{ct} that there exists an isomorphism of $\g$-modules $X^{*}(G^{\tor})_{\lbe I}\simeq \pi_{1}(G\le)_{\lbe I}/M^{\le\prime}$, where $M^{\le\prime}$ is a finite submodule of $\pi_{1}(G\le)_{\lbe I}$. Thus there exists an isomorphism of abelian groups 
\[
H^{\le n+1}_{\rm{ab}}(K_{\rm{fl}},G\e)\simeq  H^{\le n+1}(k,\pi_{1}(G)_{I})
\]
which generalizes the case $r=2$ of Theorem \ref{abt}.
\end{enumerate}
\end{remarks}

\end{document}